%% file: coalg_new_no_reduced.tex
\documentclass[10pt]{amsart}
\input{coalg-preamble}

\title{The homotopy theory of coalgebras over simplicial comonads}
\author[Hess]{Kathryn Hess}
\address[Hess]{SV UPHESS BMI, \'Ecole Polytechnique F\'ed\'erale de Lausanne}
\author[K\k{e}dziorek]{Magdalena K\k{e}dziorek}
\address[K\k{e}dziorek]{Mathematical Institute, Utrecht University}

\begin{document}

\maketitle
\begin{abstract}  We apply the Acyclicity Theorem of \cite{HKRS} (recently corrected in \cite{GKR}) to establishing the existence of model category structure on categories of coalgebras over comonads arising from simplicial adjunctions, under mild conditions on the adjunction and the associated comonad. We study three concrete examples of such adjunctions where the left adjoint is comonadic and show that in each case the component of  the derived counit  of the comparison adjunction at any fibrant object is an isomorphism, while the component of the derived unit at any 1-connected object is a weak equivalence.  To prove this last result, we explain how to construct explicit fibrant replacements for 1-connected coalgebras in the image of the canonical comparison functor from the Postnikov decompositions of their underlying simplicial sets. We also show in one case that the derived unit is precisely the Bousfield-Kan completion map. 
\end{abstract}

\tableofcontents

\section{Introduction}
Given a comonad $\K=(K, \Delta, \epsilon)$ on a model category $\sM$, it is natural to ask under what conditions the associated category $\coalg_{\K}$ of $\K$-coalgebras admits a model category structure with respect to which the (forgetful, cofree)-adjunction
$$\adjunction{\coalg_{\K}}{\sM}{U_{\K}}{F_{\K}}$$
is a Quillen pair.  If the comonad $\K$ arises from a Quillen pair
$$\adjunction{\sN}{\sM}{L}{R},$$
i.e., $\K=(LR, L\eta_{R}, \epsilon)$, where $\eta$ and $\epsilon$ are the unit and counit of the adjunction, then the existence of a model structure on $\coalg_{\K}$ enables us to formulate and study the question of homotopic codescent \cite{hess:descent} for the adjunction $L\dashv R$.  More precisely, if the model structure on $\coalg_{\K}$ is well chosen, the comparison adjunction
$$\adjunction{\sN}{\coalg_{\K},}{\mathrm{Can}_{\K}}{V_{\K}}$$
where $\mathrm{Can}_{\K}(X)=(LX, L\eta_{X})$ and $V_{\K}(Y, \delta)= \lim (\xymatrix{RY \ar@<+1ex>[r]^-{R\delta}\ar@<-0.5ex>[r]_-{\eta_{RY}} & RLRY})$, will also be a Quillen pair, so that one can seek conditions under which  the derived unit and counit are weak equivalences or, even better, under which the comparison is a Quillen equivalence.

In this article, we apply the Acyclity Theorem of \cite{HKRS} (recently corrected in \cite{GKR}, see Remark \ref{rem:corrected}) to establishing the existence of model category structure on categories of coalgebras over comonads arising from simplicial adjunctions, under mild conditions on the adjunction and the associated comonad.  We study three concrete examples of such adjunctions, where $L$ is comonadic, so that $\mathrm{Can}_{\K}$ is actually an equivalence of categories.  We show that in each case the component of  the derived counit  of the comparison adjunction at any fibrant object is an isomorphism, while the component of the derived unit at any 1-connected object is a weak equivalence.  To prove this last result, we explain how to construct explicit fibrant replacements for 1-connected coalgebras in the image of $\can_{\K}$ from the Postnikov decompositions of their underlying simplicial sets. In one case, we also show that the derived unit is precisely the Bousfield-Kan completion map.  These results refine and clarify those in \cite{BlomquistHarper1},  \cite{BlomquistHarper2}, and  \cite{BlomquistHarper3}.

We begin by introducing the three adjunctions of interest in this paper, then prove the general existence theorem for model structures on categories of coalgebras, which we apply to our three examples.

We refer the reader to Appendix \ref{coalgebras} and to \cite {MacLane} for the general theory of coalgebras over a comonad.  In Appendix \ref{appendix} we recall the notions of left- and right-induced model structures and the model structure existence results of \cite{HKRS} and \cite{GKR} that we use here.

Earlier work on special cases of model structures on categories of coalgebras can be found in  \cite{Coalgebras_Comonad_Post}, \cite {Banff1}, and \cite{HKRS}, among others.
\bigskip

\textbf{Acknowledgments.} The authors thank the referee for a careful and helpful report.

\section{A tale of three comonads and their associated coalgebras}
In this section we introduce and study three comonads $\K=(LR, \Delta, \epsilon)$ arising from adjunctions of the form 
$$\adjunction{\sset}{\sM}{L}{R},$$
where $\sset$ denotes the category of pointed simplicial sets.  
We show that the left adjoint $L$ is comonadic in each case, so that  each induced comparison adjunction
$$\adjunction{\sset}{\coalg_{\K}}{\can_{\K}}{V_{\K}}$$ 
is an equivalence of categories.

\subsection{The free abelian comonad on  $\sAb$}\label{sec:sab}
Let $\sAb$ denote the category of simplicial abelian groups. There is an adjunction
\begin{equation}\label{eqn:adj1}
\adjunction{\sset}{\sAb}{\tilZ}{\U}
\end{equation} 
where $\tilZ(X)=\mathbb{Z}(X)/\mathbb{Z}(\ast)$ is the free reduced abelian group functor, and $\U$ is the forgetful functor.  Let $\eta: \id \to \U\tilZ$ denote the unit of the adjunction.

\begin{notation} Let $A$ be a simplicial abelian group.  If $p\in (\tilZ \U A)_n$, then it is a formal integer linear combination of non-basepoint elements of level $n$ in the  simplicial set underlying $A$, which we write
$$p=\sum _{a\in A_{n}}\gamma_{a}[a],$$
where $\gamma_{a}\in \mathbb Z$ for all $a\in A_n$, and only finitely many of the $\gamma_{a}$'s are nonzero.  In terms of this notation, for every pointed simplicial set $X$, 
$$\tilZ\eta_{X}: \tilZ X \to \tilZ \U \tilZ X: \sum _{x\in X}\gamma_{x}x \mapsto \sum _{x\in X}\gamma_{x}[x].$$
\end{notation}

As recalled in Appendix \ref{coalgebras}, the adjunction (\ref{eqn:adj1}) induces a comonad $\K=(\tilZ\U, \Delta, \epsilon)$ on $\sAb$ that  factors through the category $\coalg_{\K}$ of $\K$-coalgebras, as in the diagram 

\[
\xymatrix@C=3pc {
\sset \ar@<+1ex>[rr]^{\tilZ}\ar@<-2ex>[dr]_{\can_{\K}} && \sAb \ar@<+1.1ex>[ll]^{\U} \ar@<+0.3ex>[dl]_{F_{\K}}\\
&\coalg_{\K} \ar@<-2ex>[ur]_{U_{\K}} \ar@<+0.3ex>[ul]_{V_{\K}}&
}
\]
where the left adjoints are on the outside of the diagram, $F_{\K}$ is the cofree $K$-coalgebra functor and $U_{\K}$ the forgetful functor.

Recall from Appendix \ref{coalgebras} that a $\K$-coalgebra consists of a pair $(A,\delta)$, where $A\in \sAb$, and $\delta: A\lra \tilZ\U A$ is a morphism in $\sAb$ such that for every $a\in A$
\begin{itemize}
\item (counit condition) $a=\epsilon\circ\delta(a)$,  and
\smallskip
 
\item (coassociativity condition)  $\tilZ \U \delta \circ \delta(a) = \Delta_{A}\delta(a)$.
\end{itemize}

If $(A,\delta)=(\tilZ X,\tilZ \eta_X)$ for some pointed simplicial set $X$, then $\delta$ is determined by its restriction to $X$, since it is a homomorphism of simplicial abelian groups.   Moreover, using the notation introduced above,
$$\epsilon \Big(\sum _{q\in \tilZ X} \gamma_{q} [q]\Big)= \sum _{q\in \tilZ X} \gamma_{q} q,$$
while
$$\Delta_{A} \Big(\sum _{q\in \tilZ X} \gamma_{q} [q]\Big)= \sum _{q\in \tilZ X} \gamma_{q} \big[[q]\big],$$
for every $\sum _{q\in \tilZ X} \gamma_{q} [q]\in \tilZ\U \tilZ X=\tilZ\U A$.

\begin{notation} For any $p\in \tilZ X$, we write 
$$\delta (p) = \sum _{q\in \tilZ X} \gamma_{p,q} [q].$$
\end{notation}

\begin{rem}\label{rem:delta} It follows from the counit condition that
$$p=\sum _{q\in \tilZ X} \gamma_{p,q} q,$$
while the coassociativity condition implies that
$$\sum _{q\in \tilZ X} \gamma_{p,q}\big[ [q]\big] = \sum _{q\in \tilZ X}\gamma_{p,q} \Big[\delta (q) \Big].$$
Since this is an equation in a free abelian group, and the counitality condition implies that $\delta (q) \not= \delta (q')$ if $q\not = q'$, it follows that for each $q$ such that $\gamma_{p,q}\not=0$, $\delta (q)$ must be a monomial of the form $[r]$ in $\tilZ \U \tilZ X$.  On the other hand, the counitality condition implies that if $\delta (q)=[r]$, then $r=q$.  It follows that $\delta (q)=[q]$ for every $q\in \tilZ X$ for which there exists $p\in \tilZ X$  such that $\gamma_{p,q}\not=0$.
\end{rem}

\begin{lem}\label{lem:underlying_free} If $(A,\delta)$ is a $\K$-coalgebra, then $A$ is a free simplicial abelian group.
\end{lem}

\begin{proof} The coaction $\delta$ is a map of simplicial abelian groups $\delta: A\lra \tilZ\U(A)$, which  is a monomorphism by the counitality condition, whence $A$ is a subgroup of a free abelian group and therefore free abelian.
\end{proof}

Though the underlying simplicial abelian group of any $\K$-coalgebra $(A,\delta)$ is free, the coaction map $\delta$ might be complicated. We show below, however, that $A$ admits a basis $X$ such that, with respect to this basis, the coaction map is given precisely by $\tilZ\eta_{X}$, so that there is an isomorphism of $\K$-coalgebras $(A,\delta)\cong (\tilZ X,\tilZ\eta_{X})$.

\begin{defn} We say that a free simplicial abelian group is of \emph{finite type} if it is finitely generated in each degree. A $\K$-coalgebra is of  \emph{finite type} if its underlying free simplicial abelian group is of finite type.
\end{defn}

\begin{notation}\label{notn:gp} Let $X$ be a pointed simplicial set. For any $p\in \tilZ X$, let 
$$G_p=\{q\in \tilZ X \mid \gamma_{p,q}\neq 0\}.$$

For any finite subset $W\subset \tilZ X$, let 
$$G_W=\{q\mid  \exists\, p\in W \text{ such that }\gamma_{p,q}\neq 0\}=\bigcup_{p\in W}G_{p}.$$
Let $\widehat G_{p}$ and $\widehat G_{W}$  denote the simplicial subsets of $\tilZ X$ generated by $G_{p}$ and $G_{W}$, i.e., their closures under faces and degeneracies.
\end{notation}

\begin{rem} Note that for any $p\in \tilZ X$, $G_p$ is a finite set of elements in $\tilZ X$ in the same dimension as $p$ and that  $\widehat G_{p}$ is therefore a finite set in every simplicial level.
\end{rem}

\begin{rem}\label{rem:delta-bis} Remark \ref{rem:delta} implies that $\delta(x)=[x]$ for all $x\in\widehat G_{W}$, since every element of $\widehat G_{W}$ is the image under iterated faces and degeneracies of an element of $G_{W}$, and $\delta$ is a simplicial map. 
\end{rem}

\begin{lem}\label{lem:lineary_independent}  Let $(\tilZ X, \delta)$ be a $\K$-coalgebra. For any finite subset $W\subset \tilZ X$, the simplicial set $\widehat G_W$ is dimensionwise linearly independent in $\tilZ X$, i.e., the subgroup of $\tilZ X$ generated by $\widehat G_{W}$ is $\tilZ \widehat G_{W}$.
\end{lem}

\begin{proof} Remark \ref{rem:delta-bis} implies that  if
$$0=\sum_{q\in \widehat G_{W}} \alpha_q q,$$
then
$$0=\delta\big(\sum _{q\in \widehat G_{W}}\alpha_{q} q\big)=\sum _{q\in \widehat G_{W}}\alpha_{q}[q].$$
Since the $q$'s are pairwise distinct, which implies that $\big\{[q]\mid q\in \widehat G_{W}\big\}$ is linearly independent,  it follows that $\alpha_{q}=0$ for all $q\in \widehat G_{W}$.
\end{proof}

\begin{lem}\label{lem:useful} Let $(\tilZ X, \delta)$ be a $\K$-coalgebra.  For any finite subset $W\subset \tilZ X$,
\begin{enumerate}
\item $p\in \tilZ \widehat G_W$ for all $p\in W$, and
\smallskip

\item $\delta|_{\tilZ\widehat G_{W}}=\tilZ \eta _{\widehat G_{W}}$.
\end{enumerate}
\end{lem}

Note that (2) implies that $(\tilZ\widehat G_{W}, \tilZ \eta _{\widehat G_{W}})$ is a sub $\K$-coalgebra of $(\tilZ X, \delta)$.

\begin{proof}
(1) This is a consequence of counitality, which implies that $p=\sum _{q\in G_{W}} \gamma_{p,q}q$.
\smallskip

\noindent (2) It follows from Remark \ref{rem:delta-bis} that 
$$\delta \big(\sum_{q\in \widehat G_{W}} \alpha_q q\big)=\sum_{q\in \widehat G_{W}} \alpha_q [q]=\tilZ \eta_{\widehat G_{W}}\big(\sum_{q\in \widehat G_{W}} \alpha_q q\big)$$ 
for all finite sets $\{\alpha_{q}\mid q\in \widehat G_{W}\}\subset \mathbb Z$.
\end{proof}

\begin{prop}\label{prop:filteredColim} Every $\K$-coalgebra is isomorphic to a filtered colimit of finite type $\K$-coalgebras in the image of the comparison functor $\can_{\K}: \sset \to \coalg_{\K}$ and is therefore in the essential image of $\can_{\K}$.
\end{prop}

\begin{proof} Let $(\tilZ X, \delta)$ be a $\K$-coalgebra.  Let $\mathcal P_{X}$ denote the poset of finite subsets of $\tilZ X$, which is a filtered category, and consider the functor 
$$\mathbb G: \mathcal P_{X}\to \coalg_{\K}: W \mapsto  (\tilZ \widehat G_{W}, \tilZ \eta _{\widehat G_{W}}).$$
We claim that 
$$(\tilZ X, \delta)\cong \colim_{\mathcal P_{X}} \mathbb G.$$
To prove this claim, we show that $(\tilZ X, \delta)$ satisfies the required universal property. 

Let $\iota_{W}: \mathbb G(W) \hookrightarrow (\tilZ X, \delta)$ denote the sub $\K$-coalgebra inclusion.  Given a family of $\K$-coalgebra maps $ \varphi_{W}:\mathbb G(W) \to (A', \delta')$ for all $W\in \mathcal P_{X}$ that are compatible with the inclusions $W\subset W'$, there is a  unique $\K$-coalgebra map $\widehat\varphi: (\tilZ X, \delta) \to (A', \delta')$ such that $\widehat \varphi \circ \iota_{W}= \varphi_{W}$, specified for any $p\in \tilZ X$ by $\widehat\varphi (p)= \varphi _{\{p\}} (p)$.
\end{proof}

\begin{thm}\label{thm:equivalence_of_cats} The adjunction 
\[\adjunction{\sset}{\coalg_{\K} }{\can_{\K}}{V_{\K}}\]
is an equivalence of categories.
\end{thm}

\begin{proof} First notice that since the unit of the adjunction is an isomorphism
\[
\xymatrix@C=4pc {
V_{\K}(\tilZ X,\tilZ\eta_X):= \mathrm{equal} \big( \U \tilZ X \ar@<+1ex>[r]^-{\eta_{\U\tilZ X}}\ar@<-0.5ex>[r]_-{\U\tilZ\eta_X} & \U \tilZ\U\tilZ X \big)\cong X,
}
\]
the functor $\can_{\K}$ is faithful. 
It follows from Proposition \ref{prop:filteredColim} that $\can_{\K}$ is essentially surjective and full, thus it is an equivalence of categories.

\end{proof}

\subsection{Comonads arising from suspension}
In studying the comonads introduced in this section, we apply the following classical theorem, which is the dual of the weak version of Beck's Monadicity Theorem \cite[Theorem 4.4.4]{BorceuxHandbook2}.

\begin{thm}\label{thm:DualMonadicityThm}
Let $\adjunction{\sN}{\sM}{L}{R}$ be an adjunction with associated comonad $\K=(LR, \Delta, \epsilon)$ on $\sM$.  

If the left adjoint $L$ preserves coreflexive equalizers and reflects isomorphisms, then $L$ is comonadic, i.e., the comparison adjunction 
$$\adjunction{\sN}{\coalg_{\K}}{\can_{\K}}{V_{\K}}$$ 
is an equivalence of categories.
\end{thm}

\begin{notation} Fix  a simplicial model $S^{1}$ of the circle.  Let 
$$\Sigma= -\wedge S^{1}: \sset\to \sset\text{ and }\Omega = \Map (S^{1},-): \sset \to \sset.$$
Let $S^{r}=(S^{1})^{\wedge r}.$

Let $\K_{r}$ denote the comonad associated to the adjunction $$\adjunction{\sset}{\sset}{\Sigma^r}{\Omega^r}$$ and $\K_{\infty}$ the comonad associated to the adjunction $$\adjunction{\sset}{\spe,}{\Sigma^\infty}{\Omega^\infty}$$ where $\spe$ denotes the category of symmetric spectra \cite{HoveySSSymSpectra}.
\end{notation}

\begin{prop}\label{thm:comonadic_r}
For every $r$, the left adjoint in the adjunction
\[\adjunction{\sset}{\sset}{\Sigma^r}{\Omega^r}\]
is comonadic, and thus the adjunction 
\[\adjunction{\sset}{\coalg_{\K_{r}}}{\can_{\K_{r}}}{V_{\K_{r}}}\]
is an equivalence of categories.
\end{prop}

\begin{proof}
An easy computation shows that smashing with any pointed simplicial set preserves all equalizers, so, in particular, suspension preserves coreflexive equalizers. 


It is also clear that $\Sigma^r$ reflects isomorphisms. Indeeed, suppose $f:X \lra Y$ is a morphism of pointed simplicial sets such that $\Sigma^r f$ is an isomorphism. We want to show that $f$ is an isomorphism. It is easy to see that $f$ has to be surjective. 

To verify that $f$ is also injective, suppose that there exist $x,x' \in X_{n}$ such that $f(x)=f(x')$.  Without loss of generality, we can suppose that $n\geq r$, since we can replace $x$ and $x'$ with  $s_{0}^{r-n}x$ and $s_{0}^{r-n}x'$ otherwise.
For every $a\in S_{n}^r$, 
$$\Sigma^rf[x,a]=[f(x),a]=[f(x'),a]=\Sigma^rf[x',a]$$ 
and therefore, since $\Sigma^r f$ has an inverse, $[x,a]= [x',a] \in X\wedge S^{r}$. Considering any $a$ that is not the basepoint, which must exist since $n\geq r$, we conclude that $x=x' \in X$ and thus that $f$ is a monomorphism. It follows that $f$ is an isomorphism and thus that $\Sigma^r$ reflects isomorphisms. 

Theorem \ref{thm:DualMonadicityThm} therefore implies that $\Sigma^r$ is comonadic, i.e.,equalizer that the comparison adjunction is an equivalence of categories.
\end{proof}

Since limits and colimits in the category $\spe$  of symmetric spectra are calculated objectwise, and the functor $\Sigma^\infty$ reflects isomorphisms, the next result is obtained by a similar analysis of equalizers as in the proof of Proposition \ref{thm:comonadic_r}.

\begin{cor} The left adjoint in the adjunction
\[\adjunction{\sset}{\spe}{\Sigma^\infty}{\Omega^\infty}\]
is comonadic and therefore
\[\adjunction{\sset}{\coalg_{\K_{\infty}}}{\can_{\K_{\infty}}}{V_{\K_{\infty}}}\]
is an equivalence of categories.
\end{cor}

\section{Model categories of coalgebras over comonads}

In this section we provide conditions under which the category of coalgebras over a simplicial comonad inherits a model structure from the underlying category.  We then revisit the three examples from the previous section, proving in each case that the comparison adjunction 
$$\adjunction{\sset}{\coalg_{\K}}{\can _{\K}}{V_{\K}},$$ 
which is a Quillen pair with respect to this inherited structure on the category of $\K$-coalgebras, 
is such that the components of the derived unit and counit at appropriately connected objects are weak equivalences. In case of the free abelian comonad, we upgrade this to a Quillen equivalence on the full subcategories of $1$-reduced objects and show that the derived unit map is precisely the Bousfield-Kan $\mathbb Z$-completion map.

\subsection{The general existence theorem}

The proof of the following result is an easy exercise in the theory of enriched categories.

\begin{prop} Let $\sN$ and $\sM$ be categories that are enriched and tensored over $\sset$.  If  $\adjunction{\sN}{\sM}{L}{R}$ is an $\sset$-adjunction, with associated comonad $\K$ on $\sM$, then $\coalg_{\K}$ is also enriched and tensored over $\sset$ in such a way that the forgetful functor $U_{\K}: \coalg_{\K} \lra \sM$  preserves the tensoring and the induced comparison adjunction 
$$\adjunction{\sN}{\coalg_{\K}}{\can_{\K}}{V_{\K}}$$ 
is an $\sset$-adjunction. 
\end{prop}

We refer the reader to \cite{HKRS} for an explanation of the notion of an accessible model category, which is used in the next theorem, but remark that every cofibrantly generated and locally presentable model category is accessible.

\begin{thm}\label{thm:leftLiftGeneral}
Let $\sM,\ \sN$ be accessible $\sset$-model categories.  For every simplicial Quillen adjunction 
$$\adjunction{\sN}{\sM}{L}{R}$$
with associated comonad $\K$ on $\sM$ such that the forgetful functor $U_{\K}: \coalg_{\K} \lra \sM$ takes every $\K$-coalgebra to a cofibrant object of $\sM$,
the category $\coalg_{\K}$ admits a model structure left-induced by the forgetful functor $U_{\K}: \coalg_{\K} \lra \sM$.
Moreover, the diagonal adjunction on the left of the commuting diagram
\[
\xymatrix@C=3pc {
\sN \ar@<+1ex>[rr]^{L}\ar@<-2.5ex>[dr]_{\can_{\K}} && \sM \ar@<+1.1ex>[ll]^{R} \ar@<+0.3ex>[dl]_{F_{\K}}\\
&\quad\coalg_{\K}\quad \ar@<-2ex>[ur]_{U_{\K}} \ar@<+1ex>[ul]_{V_{\K}}&
}
\]
is a simplicial Quillen pair.
\end{thm}

\begin{proof} Since the category $\coalg_{\K}$ is locally presentable by \cite[Proposition A.1.]{ChingRiehl}, we can apply Theorem \ref{thm:2.2.1}. Since the object of $\sM$ underlying any $\K$-coalgebra is  cofibrant, it is enough to show that  a good cylinder object exists for any object $X \in \coalg_{\K}$.
Because $\sM$ is a simplicial model category,  cylinder objects there can be constructed by tensoring with the cylinder object for the simplicial unit $S^0$; see Remark \ref{rem:tensored_cylinder}. 
Since $U_{\K}$ preserves tensoring with simplicial sets, good cylinder objects in $\coalg_{\K}$ can also be constructed by tensoring with the good cylinder object of $S^0$ and therefore the desired left-induced model structure on $\coalg_{\K}$ exists.

To show that the adjunction 
$$\adjunction{\sN}{\coalg_{\K}}{\can_{\K}}{V_{\K}}$$
is a Quillen pair, it is enough to check that the left adjoint preserves cofibrations and acyclic cofibrations, which is an immediate consequence of the definition of these two classes of maps in $\coalg_{\K}$ and of the commutativity of the left adjoints in the diagram above.
\end{proof}

Below we consider again the three examples  of comonads induced by simplicial Quillen pairs from the previous section,  where $\sN$ is the category $\sset$ of pointed simplicial sets endowed with the usual Kan-Quillen model structure.

\subsection{Homotopy and the free abelian comonad on $\sAb$}

Recall the comonad  $\K=(\tilZ \U,\Delta,\epsilon)$  on $\sAb$ from section \ref{sec:sab}, which gives rise to the following commuting diagram of adjunctions
\[
\xymatrix@C=3pc {
\sset \ar@<+1ex>[rr]^{\tilZ}\ar@<-2.5ex>[dr]_{\can_{\K}} && \sAb \ar@<+1.1ex>[ll]^{\U} \ar@<+0.3ex>[dl]_{F_{\K}}\\
&\quad\coalg_{\K}\quad \ar@<-2ex>[ur]_{U_{\K}} \ar@<+1ex>[ul]_{V_{\K}}&
}
\]
where the left adjoints are on the outside of the diagram.

\subsubsection{Model structure}

By Theorem \ref{thm:equivalence_of_cats}, the diagonal adjunction on the left is an equivalence of categories. Therefore there exists  a model structure on $\coalg_{\K}$ that is right-induced from the Kan-Quillen model structure on $\sset$, i.e., weak equivalences and fibrations are created by $V_{\K}$. When $\coalg_{\K}$ is equipped with this model structure, the adjunction $\can_{\K}\dashv V_{\K}$ is not only an equivalence of categories but also a Quillen equivalence, albeit of a rather uninteresting sort. 

Thanks to Lemma \ref{lem:underlying_free}, the adjunction 
\[\adjunction{\sset}{\sAb}{\tilZ}{\U}\] 
satisfies the hypotheses of Theorem \ref{thm:leftLiftGeneral}, so that the following result holds.

\begin{cor} \label{cor:left_ind_mod_str} There is a model structure on the category $\coalg_{\K}$ left-induced by the forgetful functor from the projective model structure on $\sAb$
\[\adjunction{\coalg_{\K}}{\sAb}{U_{\K}}{F_{\K}}\]
such that the adjunction
\[\adjunction{\sset}{\coalg_{\K}}{\can_{\K}}{V_{\K}}\]
is a Quillen pair.
\end{cor}

The projective model structure on $\sAb$ has weak equivalences and fibrations the underlying weak equivalences and Kan fibrations of simplicial sets. It is therefore right-induced by $\U$ from the Kan-Quillen model structure on $\sset$.

Recall that in the left-induced model structure the cofibrations and weak equivalences are created by $U_{\K}$. Retracts of limits of towers of images under $F_{\K}$ of fibrations (acyclic fibrations respectively) are fibrations (respectively, acyclic fibrations) in $\coalg_{\K}$ \cite{Banff1}.

\begin{thm}\label{thm:derived_unit_counit_we} Consider the adjunction 
\[\adjunction{\sset}{\coalg_{\K}.}{\can_{\K}}{V_{\K}}\]
If $\coalg_{\K}$ is  equipped with the left-induced model structure of Corollary \ref{cor:left_ind_mod_str} and $\sset$ with the usual Kan-Quillen model structure, then
both the component of the derived counit at any fibrant $\K$-coalgebra and the component of the derived unit at any 1-connected simplicial set are weak equivalences. 
\end{thm}

\begin{proof} The component of derived counit at any fibrant $Y$ is the same as the counit of the $\can_{\K}\dashv V_{\K}$-adjunction and thus an isomorphism. 
To show that the component of the derived unit at a 1-connected simplicial set $X$ is a weak equivalence, we construct an explicit fibrant replacement of $\can_{\K}(X)$ in $\coalg_{\K}$ from a Postnikov tower  for $X$ built out of simplicial abelian groups and fibrations (of simplicial abelian groups) between them. Note that since $\can_{\K}$ is an equivalence of categories, it preserves limits.

For any abelian group $A$, there is a simplicial model $K(A,n)$ of the Eilenberg-MacLane space of type $(A,n)$,  given by the levelwise tensor product $A\otimes \tilZ(S^n)$.  Note that $K(A,n)$ is the simplicial set underlying a simplicial abelian group.

Consider the usual Postnikov tower of $X$. Since $X$ is 1-connected, we can start at level $2$, where we have
$$X\langle2\rangle:=K(\pi_2(X),2).$$

The next level is given by a pullback square
\[
\xymatrix@C=3pc{X\langle3\rangle \ar[r] \ar[d] & PK(\pi_3(X),4) \ar[d] \\
X\langle2\rangle \ar[r] & K(\pi_3(X),4) }
\]
where the right hand map is a Kan fibration underlying a fibration of simplicial abelian groups, since the path space is obtained from a factorization (in $\sAb$)
\[
\xymatrix@C=3pc{0\ \ar@{^{(}->}[r]^(.3){\sim} & PK(\pi_3(X),4) \ar@{->>}[r] & K(\pi_3(X),4).}
\]

The procedure continues in a similar manner, giving rise to a tower of maps

\[
\xymatrix@C=2pc{\cdots \ar@{->}[r] \  X\langle4\rangle \ar@{->}[r] & X\langle3\rangle \ar@{->}[r] & X\langle2\rangle \\
& & X\ar[ull] \ar[ul]  \ar[u] }
\]
where every horizontal map is the pullback of a Kan fibration underlying a fibration of simplicial abelian groups, and there is a weak equivalence $X \lra \lim_n X\langle n\rangle$.
Applying $\can_{\K}$ to this tower, we obtain a diagram

\[
\xymatrix@C=2pc{\cdots \ar@{->}[r] \can_{\K}X\langle4\rangle \ar@{->}[r] & \can_{\K} X\langle3\rangle \ar@{->}[r] & \can_{\K}X\langle2\rangle\\
& & \can_{\K} X \ar[ull] \ar[ul]  \ar[u]}
\]
in $\coalg_{\K}$ where $\lim_n \can _{\K} X\langle n\rangle\cong\can _{\K} (\lim_n X\langle n\rangle)$, since $\can_{\K}$ is an equivalence of categories.  It follows that the natural map 
$$\can_{\K} X \to \lim_n \can _{\K} X\langle n\rangle$$
is a weak equivalence because $\can _{\K}$ is left Quillen and therefore preserves all weak equivalences, since all objects in $\sset$ are cofibrant.

Note that $\can_{\K} X\langle2\rangle$ is fibrant  in $\coalg_{\K}$, since it is in the image of $F_{\K}$. Similarly,  by construction the horizontal maps are fibrations in $\coalg_{\K}$, since they are pullbacks of maps arising from applying $F_{\K}$ to fibrations in $\sAb$. Since $\can _{\K}$ preserves all weak equivalences and limits, it follows that $\can_{\K} X \lra \can_{\K} (\lim_n X\langle n\rangle)$ is a fibrant replacement in $\coalg_{\K}$ for $\can _{\K}X$.  A model of the derived unit of the $\can_{\K}\dashv V_{\K}$-adjunction for our simply connected $X$ is therefore given by
$$X\xrightarrow \cong V_{\K}\can_{\K}X \to V_{\K}\can_{\K}(\lim_n X\langle n\rangle)\cong \lim_n X\langle n\rangle,$$
which we already know is a weak equivalence.
\end{proof}

We can upgrade the result above to a Quillen equivalence between the categories of $1$-reduced simplicial sets and $1$-reduced $\K$-coalgebras. 
First notice that since the categories of $1$-reduced simplicial sets ($\sset_1$), $1$-reduced simplicial abelian groups ($\sAb_1$), and $1$-reduced $K$-coalgebras ($\coalg_{\K,1}$) are locally presentable, and the inclusion functors 
$$\sset_1 \lra \sset$$
$$\sAb_1 \lra \sAb$$
$$\coalg_{\K,1} \lra \coalg_{\K}$$
preserve all limits and colimits, they have both left and right adjoints. Moreover the right adjoints restrict to identity functors on the full subcategories of $1$-reduced objects.

Since all the functors on $\sset$, $\sAb$, and $\coalg_{\K}$ are defined degreewise, there is an induced diagram of functors restricted to the $1$-reduced objects.

\[
\xymatrix@C=3pc {
\sset_1 \ar@<+1ex>[rr]^{\tilZ}\ar@<-2ex>[dr]_{\can_{\K}} && \sAb_1 \ar@<+1.1ex>[ll]^{\U} \ar@<+0.3ex>[dl]_{F_{\K}}\\
&\qquad\coalg_{\K, 1}\qquad\ar@<-2ex>[ur]_{U_{\K}} \ar@<+0.3ex>[ul]_{V_{\K}}&
}
\]

\begin{lem}\label{lem:left_ind_to_1red} There exist model structures on the full subcategories of $1$-reduced objects in $\sM$ left--induced by the inclusion functor $\sM_1 \lra \sM$, when $\sM$ is $\sset,\ \sAb,$ or  $\coalg_{\K}$, with respect to the Kan-Quillen model structure on $\sset$, the usual (projective) model structure on $\sAb$, and the model structure of Corollary \ref{cor:left_ind_mod_str}  on $\coalg_{\K}$. 
\end{lem}

\begin{proof}Since all the categories in question are locally presentable, we can use the methods of \cite{HKRS} to lift the model structures.

In cases of $\sset$ and $\coalg_{\K}$, we apply Theorem \ref{thm:2.2.1}. Note that in both cases if $X$ is $1$-reduced, then the good cylinder object for $X$ given by tensoring with the good cylinder object for $S^0$ is a construction in the subcategory of $1$-reduced objects (Remark \ref{rem:tensored_cylinder}). This follows from the fact that when a $1$-reduced simplicial set is smashed with any simplicial set, the result is $1$-reduced, since the unique element in degree $1$ of a $1$-reduced simplicial set is the degeneracy of the basepoint in degree $0$. Thus in both cases there exists a good cylinder object for $X$ in the subcategory of $1$-reduced objects.
Since all objects in $\sset$ and $\coalg_{\K}$ are cofibrant, we can conclude for $\sM=\sset$ and  $\sM=\coalg_{\K}$.

The left-induced model structure on $\sAb_1$ exists by Theorem \ref{thm:square}, since there is a square of adjunctions 
\[ \xymatrix@R=4pc@C=4pc{ \sset \ar@{}[r]|-{\perp} \ar@<-1ex>[d]_-{\tilZ} \ar@<-1ex>[r] & \sset_1  \ar@<1ex>[d]^-{\tilZ} \ar@<-1ex>[l]_-i \\  
\ar@{}[r]|-{\top} \ar@{}[u]|-{\dashv} {\sAb} \ar@<1ex>[r] \ar@<-1ex>[u]_-{\U} & \sAb_1 \ar@<1ex>[u]^-{\U} \ar@<1ex>[l]^-i\ar@{}[u]|-{\vdash} }\] 
where the left adjoints commute and $i\U\cong \U i$, while the previous argument provides a left-induced model structure on $\sset_1$. In fact  Theorem \ref{thm:square} provides two (potentially identical) lifted model structures on $\sAb_1$, one right- and one left-lifted. 
\end{proof}

\begin{thm} The adjunction 
\[\adjunction{\sset}{\coalg_{\K}}{\can_{\K}}{V_{\K}}\]
restricts to an adjunction
\[\adjunction{\sset_1}{{\coalg_{\K}}_1,}{\can_{\K}}{V_{\K}}\]
which is an equivalence of categories and a Quillen equivalence, when both categories are equipped with the left-lifted model structure of Lemma \ref{lem:left_ind_to_1red}.
\end{thm}

\begin{proof} The restricted adjunction is clearly an equivalence of categories, since the unrestricted adjunction is. It is a Quillen pair, because the left adjoint preserves cofibrations and acyclic cofibrations. It remains to show that the derived unit and derived counit are weak equivalences, which follows from Theorem \ref{thm:derived_unit_counit_we} and its proof, where $X$ is $1$-reduced (so in particular 1-connected). Notice that all the constructions made in the proof of Theorem \ref{thm:derived_unit_counit_we} for $1$-reduced $X$ live in the categories of $1$-reduced simplicial sets,  $1$-reduced abelian groups, and $1$-reduced $K$-coalgebras, respectively.
\end{proof}

\subsubsection{Bousfield-Kan completion}

We now show that the derived unit of the adjunction \[\adjunction{\sset}{\coalg_{\K}}{\can_{\K}}{V_{\K}}\] is exactly the Bousfield-Kan $\Z$-completion map.
The proof requires a different fibrant replacement in $\coalg_{\K}$ from that of the proof of Theorem \ref{thm:derived_unit_counit_we}, which we construct as follows, using the restricted (or ``fat'') totalization functor $\Tot^{\res}$ \cite[Definition 1.9]{AroneChing} defined for a semicosimplicial object $Z^\bullet$ in a simplicial model category $\M$ by
\[\Tot^{\res}(Z^\bullet)=\Hom_{\Delta_{\inj}}(\Delta^\bullet,Z^\bullet).\]
Here $\Delta_{\inj}$ denotes the wide subcategory of the ordinal category $\Delta$ where the morphisms are the injections. As in the case of totalization of a cosimplicial object \cite[Section VII. 5]{GoerssJardine}, the restricted totalization of a semi-cosimplicial object $Z^\bullet$  is the limit of a tower 
$$\cdots \to \Tot^{\res}_{n+1}Z^\bullet \to  \Tot^{\res}_n Z^\bullet\to \Tot^{\res}_{n-1} Z\to \cdots \to \Tot^{\res}_0 Z^\bullet,$$ 
where $\Tot^{\res}_0 Z^\bullet \cong Z^0$ and for any $n>0$, there is a pullback diagram 
\[ \xymatrix{ \Tot^{\res}_n Z \ar[d] \ar[r] \ar@{}[dr]|(.2){\lrcorner} & Z_n^{\Delta^n} \ar[d]^{p_n} \\ \Tot^{\res}_{n-1} Z \ar[r] & Z_n^{\partial\Delta^n},}\]
where $p_{n}$ is induced by the inclusion of $\partial \Delta^{n}$ into $\Delta^{n}$.  In particular, there is a natural map $p:\Tot^{\res}(Z^\bullet) \to Z^{0}$. 
Moreover, if $Z^{\bullet}$ is objectwise fibrant in $\M$, then all of the maps $\Tot^{\res}_n Z^\bullet\to \Tot^{\res}_{n-1} Z$ are fibrations  in $\M$, whence $p:\Tot^{\res}(Z^\bullet) \to Z^{0}$ is a fibration as well.  Finally, if $\eta: Z^{-1}\to Z^{\bullet}$ is a coaugmentation, then there is a factorization in $\M$
$$\xymatrix{ Z^{-1} \ar[rr]^\eta \ar[dr]_{\widehat \eta} ^\sim & & Z^{0} \\ & \Tot^{\res} Z^{\bullet} \ar[ur]_{p}}$$
in which $\hat\eta$ is a weak equivalence by the usual ``extra codegeneracy'' argument.

We now apply the restricted totalization functor in the case $\M= \coalg_{\K}$, endowed with its simplicial model structure left-induced from the projective model structure on $\sAb$ by 
\[\adjunction{\coalg_{\K}}{\sAb}{U_{\K}}{F_{\K}}.\] 
To simplify notation, let $K=\tilZ \U$.  Recall that $\Delta:K \to KK$ and $\ve: K\to \Id$ denote the comultiplication and counit of the comonad $\K$.  

For any $(Y,\rho) \in \coalg_\K$, 
consider the following coaugmented semi-cosimplicial object in $\coalg_\K$ 
\[ \xymatrix{ (Y,\rho) \ar[r]^(0.3)\rho & C^{\bullet}(Y,\rho)=\Big(  F_{\K}Y \ar@<1ex>[r]^-{K\rho} \ar@<-1ex>[r]_(0.6){\Delta_{Y}} & F_{\K}(KY)\ar@<2.5ex>[r]^{K^{2}\rho} \ar [r]^{K\Delta_{Y}}\ar@<-1.5ex>[r]_{\Delta_{KY}} & \cdots\Big).}\] 
As seen in the general case above, there exists a factorization in $\coalg_\K$ 
\[ \xymatrix{ (Y,\rho) \ar[rr]^\rho \ar[dr]_{\widehat\rho}^{\sim} & & F_\K Y \\ & \Tot^{\res} C^{\bullet} (Y,\rho) \ar@{->>}[ur]_{p}}\]
where $p$ is a fibration because every $\K$-coalgebra in the image of $F_{\K}$ is fibrant in the left-induced model structure on $\coalg_{\K}$, which implies that $C^{\bullet} (Y,\rho)$ is objectwise fibrant.
It follows that $\widehat \rho: (Y,\rho) \to \Tot^{\res}C^{\bullet}(Y,\rho)$ is a fibrant replacement in $\coalg_\K$ for every $(Y,\rho)$.

\begin{prop}The derived unit of the adjunction \[\adjunction{\sset}{\coalg_{\K}}{\can_{\K}}{V_{\K}}\]  is the Bousfield-Kan $\Z$-completion map.
\end{prop}

\begin{proof} The construction of the fibrant replacement above implies that a model for the component of  the derived unit map of the adjunction $\can_\K\dashv V_\K$ at a pointed simplicial set $X$ is given by the composite
\[\xymatrix{X \ar[r]^(.3){\cong} & V_\K \can_\K (X) \ar[r] & V_\K \Tot^{\res} C^{\bullet}(\can_\K X).}
\]
Since $V_{\K}$ is a simplicial right adjoint, and $\can_\K\dashv V_{\K}$ is an equivalence of categories, 
$$V_\K \Tot^{\res} C^{\bullet}(\can_\K X) {\cong}  \Tot^{\res}V_\K \can_{\K} {\widetilde{C}^{\bullet}(X)} \cong  \Tot^{\res}\widetilde{C}^{\bullet}(X),$$
where $\widetilde{C}^{\bullet}(X)$ is the following coaugmented restricted-cosimplicial object in $\sset$:
\[ \xymatrix{ X  \ar[r]^(0.3){\eta_X} &\widetilde{C}^{\bullet}(X)=( \U\tilZ X \ar@<1ex>[r]^-{\eta_{\U \tilZ X}} \ar@<-1ex>[r]_-{\U\tilZ\eta_X} & \U\tilZ\U\tilZ X & \cdots).}\] 
This identification relies on the isomorphism $\can_{\K }(\U\tilZ X)\cong F_{\K}(\tilZ X)$ of $\K$-coalgebras for every pointed simplicial set $X$. 

Notice that $\widetilde{C}^{\bullet}(X)$ is the restriction of a cosimplicial object, the cobar construction associated to the monad on $\sset$ with underyling functor $\U\tilZ$:
\[\xymatrix{ X  \ar[r]^(0.3){\eta_X} & \overline{C}^{\bullet}(X)=(\U\tilZ X \ar@<1ex>[r]^-{\eta_{\U \tilZ X}} \ar@<-1ex>[r]_-{\U\tilZ\eta_X} & \U\tilZ\U\tilZ X \ar[l] & \cdots).}\] 
Since $\overline{C}^{\bullet}(X)$ is Reedy fibrant in $\sset^\Delta$ \cite[Example X. 4.10(ii)]{BousfieldKan}, the natural map $ \Tot^{\res}\widetilde{C}(X) \to \Tot \overline{C}(X) $ is a weak equivalence \cite{AroneChing}.

By \cite[Section I.4.2]{BousfieldKan}, the map $X\to \Tot \overline{C}(X)$ induced by $\eta_{X}$ is a model for the $\Z$-completion of $X$ and therefore $X\to V_\K \Tot^{\res} C^{\bullet}(\can_\K X)$ is as well.\end{proof}

As in the previous section, here we used repeatedly the fact that $\can_\K$ is an equivalence of categories. In particular, if $\can_{\K}$ were not an equivalence of categories, then we would know little about how to compute limits in $\coalg_\K$.

\subsection{Homotopy and the comonads arising from suspension}

In this section we concentrate on analyzing the two adjunctions
\[\adjunction{\sset}{\sset}{\Sigma^r}{\Omega^r}\]
and 
\[\adjunction{\sset}{\spe}{\Sigma^\infty}{\Omega^\infty}\]
and the corresponding comonads, $\K_{r}$ and $\K_{\infty}$, and their associated categories and adjunctions.

\[
\xymatrix@C=3pc {
\sset \ar@<+1ex>[rr]^{\Sigma^r}\ar@<-2.5ex>[dr]_{\can_{\K_{r}}} && \sset \ar@<+1.1ex>[ll]^{\Omega^r} \ar@<+0.3ex>[dl]_{F_{\K_r}}\\
&\qquad\coalg_{\K_{r}}\qquad \ar@<-2ex>[ur]_{U_{\K_{r}}} \ar@<+0.8ex>[ul]_{V_{\K_{r}}}&
}
\]

\[
\xymatrix@C=3pc {
\sset \ar@<+1ex>[rr]^{\Sigma^\infty}\ar@<-2.5ex>[dr]_{\can_{\K_{\infty}}} && \spe \ar@<+1.1ex>[ll]^{\Omega^\infty} \ar@<+0.3ex>[dl]_{F_{\K_\infty}}\\
&\qquad\coalg_{\K_{\infty}}\qquad \ar@<-2ex>[ur]_{U_{\K_{\infty}}} \ar@<+0.8ex>[ul]_{V_{\K_{\infty}}}&
}
\]

As in the previous section, one could use the equivalence of categories $V_{\K_{r}}$ or $V_{\K_{\infty}}$ to right-induce model structures on the categories $\coalg_{\K_{r}}$ and $\coalg_{\K_{\infty}}$ from the Kan-Quillen model structure on $\sset$, with respect to which the comparison adjunctions are trivially Quillen equivalences. 
Mimicking the more interesting result for $\sAb$, we proceed instead with left-inducing  model structures using Theorem \ref{thm:leftLiftGeneral}, of which the following existence results are immediate consequences.

\begin{cor}\label{thm:mod_str_r} There is a model structure on a category $\coalg_{\K_{r}}$ left induced by the forgetful functor $U_{\K_{r}}$ from the Kan-Quillen model structure on $\sset$ \[\adjunction{\coalg_{\K_{r}}}{\sset}{U_{\K_{r}}}{F_{\K_r}}\] such that the induced adjunction \[\adjunction{\sset}{\coalg_{\K_{r}}}{\can_{\K_r}}{V_{\K_{r}}}\] is a Quillen pair. \end{cor}

\begin{cor}\label{thm:mod_str_infty}There is a model structure on a category $\coalg_{\K_{\infty}}$ left induced by the forgetful functor $U$ from the projective stable model structure on $\spe$
\[\adjunction{\coalg_{\K_{\infty}}}{\spe}{U}{F_{\K_\infty}}\]
such that the adjunction
\[\adjunction{\sset}{\coalg_{\K_{\infty}}}{\can_{\K_{\infty}}}{V_{\K_{\infty}}}\]
is a Quillen pair.
\end{cor}

Given these model structures, we can study the derived unit and counit of the comparison adjunction.

\begin{thm}\label{thm:derived_unit_counit_we_Sigma_r} Consider the adjunction 
\[\adjunction{\sset}{\coalg_{\K_{r}}}{\can_{\K_r}}{V_{\K_{r}}}\]
where the model structure on $\coalg_{\K_{r}}$ is the left-induced model structure of Theorem \ref{thm:mod_str_r}, and the model structure on $\sset$ is the usual Kan-Quillen model structure.
The component of the derived counit at any fibrant $\K_{r}$-coalgebra is a weak equivalence, as is the component of the derived unit at any 1-connected simplicial set. 
\end{thm}

\begin{proof} The proof follows exactly the same lines as the proof of Theorem \ref{thm:derived_unit_counit_we}. The key is that one can build a Postnikov tower for $1$-connected simplicial sets as before, and any $K(\pi_n(X),m)$ space can be constructed as $\Omega^r(K(\pi_n(X),m+r))$, which is fibrant. Moreover the model for the path space $PK(\pi_n(X),m)$ needed for constructing the next level of the Postnikov tower can be also taken as $\Omega^r(PK(\pi_n(X),m+r))$, since $\Omega^r$ preserves cotensoring with simplicial sets and  is a right Quillen functor  and therefore preserves fibrations and weak equivalences between fibrant objects.  Similarly, the fibration used in the pullback diagram for the next level of Postnikov tower is in the image of $\Omega^r$:
$$\Omega^r(PK(\pi_n(X),m+r)) \lra \Omega^r(K(\pi_n(X),m+r)).$$

It follows that for any $1$-connected simplicial set $X$,  $\can_{\K_{r}} (\lim_n X\langle n\rangle)$ is a fibrant replacement in $\coalg_{\K_r}$ for $\can_{\K_{r}}(X)$ by the same argument as in the proof of Theorem \ref{thm:derived_unit_counit_we}.  Thus the component of the derived unit at any 1-connected simplicial set $X$ is indeed a weak equivalence.
\end{proof}

\begin{thm}\label{thm:derived_unit_counit_we_Sigma_infty} Consider the adjunction 
\[\adjunction{\sset}{\coalg_{\K_{\infty}}}{\can_{\K_{\infty}}}{V_{\K_{\infty}}}\]
where the model structure on $\coalg_{\K_{\infty}}$ is the left--induced model structure of Theorem \ref{thm:mod_str_infty}, and the model structure on $\sset$ is the usual Kan-Quillen model structure.
The component of the derived counit at any fibrant $\K_{\infty}$-coalgebra is a weak equivalence, as is the component of the derived unit at any 1-connected simplicial set. 
\end{thm}

\begin{proof} A proof analogous to the one of Theorem \ref{thm:derived_unit_counit_we_Sigma_r} works. The only difference now is that the $K(\pi_n(X),m)$  are given by $\Omega^\infty$ applied to Eilenberg-MacLane spectra. The same is true for the path objects and fibrations used to build the Postnikov tower for a $1$-connected simplicial set $X$.
\end{proof}

\appendix
\section{Comonads and their coalgebras}\label{coalgebras}

A \emph{comonad} on a category $\sM$ consists of an endofunctor $K:\sM\to \sM$, together with natural transformations $\Delta:K\to K\circ K$ and $\epsilon:K\to \id_{\sM}$ such that $\Delta$ is appropriately coassociative and counital., i.e., $\K=(K,\Delta, \epsilon)$ is a comonoid in the category of endofunctors of $\sM$.

  If $\adjunction{\sN}{\sM}{L}{R}$ is a pair of adjoint functors, with unit $\eta:\id_{\sM}\to RL$ and counit $\epsilon:LR\to \id_{\sM}$, then $(LR, L\eta_{R}, \epsilon)$ is a comonad on $\sM$.

\begin{defn}\label{defn:Kcoalg}  Let $\K=(K, \Delta, \ve)$ be a comonad on $\sM$.  The objects of the \emph{Eilenberg-Moore category of $\K$-coalgebras}, denoted $\coalg_{\K}$,  are pairs $(Y, \delta)$, where $Y$ is an object in $\sM$, and $\delta\in \sM(Y, KY)$ and satisfies
$$K\delta \circ \delta = \Delta_{Y}\circ \delta\quad\text{and}\quad \epsilon_{Y}\circ \delta =\id_{Y}.$$ 
A morphism in $\coalg_{\K}$ from $(Y,\delta)$ to $(Y',\delta')$ is a morphism $f:Y\to Y'$ in $\sM$ such that $Kf\circ \delta=\delta'\circ f$.
\end{defn}

The category $\coalg_{\K}$ of $\K$-coalgebras is related to the underlying category $\sM$ as follows.

\begin{rem} \label{rmk:K-adjunct} Let $\K=(K, \Delta, \epsilon)$ be a comonad on $\sM$. The forgetful functor $U_{\K}:\coalg_{\K}\to \sM$ admits a right adjoint
$$F_{\K}:\sM\to \coalg_{\K},$$
called the \emph{cofree $\K$-coalgebra functor}, which is defined on objects by
$$F_{\K}(Y) = (KY, \Delta_{Y})$$
and on morphisms by
$$F_{\K}(f)=Kf.$$
Note that $\K$ itself is the comonad associated to the $(U_{\K},F_{\K})$-adjunction.
\end{rem}

If the comonad $\K$ arises from an adjunction $\adjunction{\sN}{\sM}{L}{R}$, then a comparison functor, defined below, mediates between $\sN$ and $\coalg_{\K}$.

\begin{defn} Let $\adjunction{\sN}{\sM}{L}{R}$ be a pair of adjoint functors, with unit $\eta:\id_{\sN}\to RL$ and counit $\epsilon: LR\to \id_{\sM}$.  Let $\K$ denote the associated comonad.  The \emph{canonical $\K$-coalgebra functor}
$$\can _{\K}:\sN\to \coalg_{\K}$$
is defined on objects by
$$\can _{\K}(X)=(LX,L\eta_{X})$$
and on morphisms by
$$\can _{\K}(f)=Lf.$$
If $\sN$ admits equalizers, then $\can _{\K}$ has a right adjoint, the ``primitives'' functor
$$V_{\K}:\coalg_{\K}\to \sM,$$
which is defined on objects by 
$$V_{\K}(Y, \delta)=\operatorname{equal} (RY \egal {R\delta}{\eta_{RY}} RLRY).$$

The functor $L$ is \emph{comonadic} if $\can _{\K}$ is an equivalence of categories.
\end{defn}

\section{Model category techniques}\label{appendix}
In this appendix we recall techniques from \cite{HKRS} for establishing the existence of induced model category structures.

\begin{notation}\label{notation:cof} For any class of maps $X$ in a category $\sM$, we let $\LLP(X)$ (respectively, $\RLP(X)$) denote the class of maps having the left lifting property (respectively, the right lifting property) with respect to all maps in $X$. We use notation $X$-cof for the class of maps $\LLP(\RLP(X))$.\end{notation}

\begin{defn} A \emph{weak factorization system} on a category $\sfC$ consists of a pair $(\mathcal L,\mathcal R)$ of classes of maps so that the following conditions hold.
\begin{itemize}
\item Any morphism in $\sfC$ can be factored as a morphism in $\mathcal L$ followed by a morphism in $\mathcal R$.
\item $\mathcal L = \LLP(\mathcal R)$ and $\mathcal R = \RLP(\mathcal L)$.
\end{itemize}
\end{defn}

In particular, if $(\sM,\Fib,\Cof,\WE)$ is a model category, then $(\Cof \cap \WE, \Fib)$ and $(\Cof, \Fib \cap \WE)$ are weak factorization systems.  If one additional condition is satisfied, the converse holds as well.

\begin{prop}[{Joyal and Tierney} {\cite[7.8]{joyal-tierney}}]\label{prop:model-via-wfs} If $\sM$ is a bicomplete category, and $\Fib, \Cof, \WE$ are classes of morphisms so that 
\begin{itemize}
\item $\WE$ satisfies the 2-of-3 property, and 
\item $(\Cof \cap \WE, \Fib)$ and $(\Cof, \Fib \cap \WE)$ are weak factorization systems,
\end{itemize}
then $(\sM,\Fib,\Cof,\WE)$ is a model category.
\end{prop}

\begin{defn}
Let $(\sM,\Fib,\Cof,\WE)$ be a model category and consider a pair of adjunctions
\[ \xymatrix@C=4pc@R=4pc{ \sN \ar@<1ex>[r]^V \ar@{}[r]|\perp & \sM \ar@<1ex>[l]^R \ar@<1ex>[r]^L\ar@{}[r]|\perp & \sfC \ar@<1ex>[l]^U}\]
where the categories $\sfC$ and $\sN$ are bicomplete.
If they exist:
\begin{itemize}
\item the \emph{right-induced model structure} on $\sfC$ is given by \[\Big(\sfC,U^{-1}\Fib, \LLP\big({U^{-1}(\Fib \cap\WE)}\big), U^{-1}\WE\Big),\] and
\item the \emph{left-induced model structure} on $\sN$ is given by \[ \Big(\sN, \RLP{\big(V^{-1}(\Cof\cap\WE)\big)}, V^{-1}\Cof, V^{-1}\WE\Big).\]
\end{itemize}
\end{defn}

\begin{rem} The adjunction $(L,U)$ is a Quillen pair with respect to the right-induced model category structure on $\sfC$, when it exists.  Dually, the adjunction $(V,R)$ is a Quillen pair with respect to the left-induced model category structure on $\sN$, when it exists.
\end{rem}

Establishing an induced model category structure therefore reduces to proving the existence of appropriate weak factorization systems and checking a certain acyclicity condition.

\begin{prop}\label{prop:acyclicity-reduction}\cite[Proposition 2.1.4]{HKRS} Suppose $(\sM, \Fib, \Cof,\WE)$ is a model category, $\sfC$ and $\sN$ are bicomplete categories, and there exist adjunctions 
\[ \xymatrix@C=4pc@R=4pc{ \sN \ar@<1ex>[r]^V \ar@{}[r]|\perp & \sM \ar@<1ex>[l]^R \ar@<1ex>[r]^L\ar@{}[r]|\perp & \sfC \ar@<1ex>[l]^U}\] so that the right-induced weak factorization systems exists on $\sfC$, and the left-induced weak factorization systems exists on $\sN$.  It follows that
\begin{enumerate}
\item the right-induced model structure exists on $\sfC$ if and only if \[\LLP\,(U^{-1}\Fib) \subset U^{-1}\WE;\] and
\item the left-induced model structure exists on $\sN$ if and only if \[\RLP{(V^{-1}\Cof)} \subset V^{-1}\WE.\]
\end{enumerate}
\end{prop}

Under reasonable conditions on the categories $\sM$, $\sfC$, and $\sN$, the desired right- and left-induced weak factorization systems are guaranteed to exist, so that only the acyclicity conditions remains to be checked.

\begin{rem}\label{rem:corrected}
The following corollary first appeared in \cite{HKRS}, however the claims that it was based on, namely Theorems 3.3.1 and 3.3.2 of \cite{HKRS} are not correct. The reason is very subtle, and has to do with what exactly is lifted in applying \cite[Proposition 13]{BourkeGarner}. This was recently fixed by \cite[Theorem 2.6]{GKR}, thus the corollary follows making all the further results of \cite{HKRS} recalled here true.
\end{rem}

\begin{cor}\label{cor:cofib-gen}\cite[Corollary 2.7]{GKR} Suppose $(\sM, \Fib, \Cof,\WE)$ is a locally presentable, cofibrantly generated model category, $\sfC$ and $\sN$ are locally presentable categories, and there exist adjunctions 
\[ \xymatrix@C=4pc@R=4pc{ \sN \ar@<1ex>[r]^V \ar@{}[r]|\perp & \sM \ar@<1ex>[l]^R \ar@<1ex>[r]^L\ar@{}[r]|\perp & \sfC. \ar@<1ex>[l]^U}\] 
It follows that
\begin{enumerate}
\item the right-induced model structure exists on $\sfC$ if and only if \[\LLP\,(U^{-1}\Fib) \subset U^{-1}\WE;\] and
\item the left-induced model structure exists on $\sN$ if and only if \[\RLP{(V^{-1}\Cof)} \subset V^{-1}\WE.\]
\end{enumerate}
\end{cor}

The following consequences of Corollary \ref{cor:cofib-gen} are frequently applied in this paper. The first one is a simplified version of \cite[Theorem 2.2.1]{HKRS}

\begin{thm}\label{thm:2.2.1} Consider an adjunction between locally presentable categories
\[ \xymatrix@C=4pc@R=4pc{ \sN \ar@<1ex>[r]^V \ar@{}[r]|\perp & \sM \ar@<1ex>[l]^R } \]
where $\sM$ is a cofibrantly generated model category. If for any $X$ in $\sN$, $V(X)$ is cofibrant in $\sM$, and a good cylinder object for $X$ exists in $\sN$, i.e., there is a factorization in $\sN$

\[\xymatrix@C=4pc@R=4pc{X\coprod X \lra \mathrm{Cyl}(X) \lra X}
\]
 of the fold map, where the first map is sent by $V$ to a cofibration in $\sM$ and the second map to a weak equivalence in $\sM$, then $\sN$ admits a model structure left-induced by the adjunction $V\dashv R$.
\end{thm}

\begin{rem}\label{rem:tensored_cylinder} If both categories $\sN$ and $\sM$ are simplicially enriched and tensored over $\sset$, so that the category $\sM$ is a simplicial model category, and the functor $V$ preserves tensoring with $\sset$, then the factorization in the theorem above is obtained as in \cite[Theorem 2.2.3]{HKRS}, i.e., by tensoring the object $X$ with the cylinder factorisation for $S^0$ in $\sset$.
\end{rem}

\begin{thm}\label{thm:square}\cite[Theorem 2.3.2]{HKRS} Given a square of adjunctions 
\[ \xymatrix@R=4pc@C=4pc{ \sN \ar@{}[r]|-{\perp} \ar@<-1ex>[d]_-L \ar@<-1ex>[r]_-{R} & \sM  \ar@<1ex>[d]^-L \ar@<-1ex>[l]_-V \\  
\ar@{}[r]|-{\top} \ar@{}[u]|-{\dashv} \sfC \ar@<1ex>[r]^-{R} \ar@<-1ex>[u]_-U & \sP \ar@<1ex>[u]^-U \ar@<1ex>[l]^-V\ar@{}[u]|-{\vdash} }\] 
between locally presentable categories, suppose that $(\sN,\Cof_\sN,\Fib_\sN,\WE_\sN)$ is a model category such that the left-induced model structure  created by $V$, denoted $(\sM,\Cof_\sM,\Fib_\sM,\WE_\sM)$, and the right-induced model structure created by $U$, denoted $(\sfC,\Cof_\sfC,\Fib_\sfC,\WE_\sfC)$ both exist.

If $UV\cong VU$, $LV\cong VL$ (or, equivalently, $UR\cong RU$), 
then there exists a right-induced model structure on $\sP$, created by $U\colon \sP \lra \sM$, and a left-induced model structure on $\sP$, created by $V\colon \sP \lra \sfC$, so that the identity is a left Quillen functor from the right-induced model structure to the left-induced one:
\[ \xymatrix@C=4pc@R=4pc{  \sP_{\mathrm{right}} \ar@<1ex>[r]^-\id \ar@{}[r]|-\perp & \sP_{\mathrm{left}}. \ar@<1ex>[l]^-\id}\]
\end{thm}

\begin{rem} All the results above can be generalized from cofibrantly generated model categories to \emph{accessible} model categories in the sense of \cite{RosickyAccessibleModelCats}. Recall that a model structure on a locally presentable category is accessible if its functorial factorizations are given by accessible functors. Lifting results for accessible model structures are discussed in detail in \cite{HKRS} and \cite{GKR}.
\end{rem}

\bibliographystyle{plain}
\bibliography{biblio}

\end{document}

%% file: coalg-preamble.tex
\usepackage{amssymb, amsmath, amsthm} 
\usepackage[T1]{fontenc}
\usepackage{mathrsfs,comment}
\usepackage{graphicx}
\usepackage{placeins}
\usepackage{rotating}
\usepackage{tikz}
\usepackage{float}
\usepackage{subfigure}
\usepackage{hvfloat}
\usepackage{caption}
\usepackage{pdflscape}
\usepackage{hyperref}  
\usepackage{url}
\usepackage[all,arc,2cell]{xy}
\UseAllTwocells
\usepackage{enumerate}
\usepackage{chngcntr}
 \usepackage{lineno}
 \usepackage{blindtext}
\usepackage{verbatim}

\hypersetup{%
  bookmarksnumbered=true,%
  bookmarks=true,%
  colorlinks=true,%
  linkcolor=blue,%
  citecolor=blue,%
  filecolor=blue,%
  menucolor=blue,%
  pagecolor=blue,%
  urlcolor=blue,%
  pdfnewwindow=true,%
  pdfstartview=FitBH}

\def\@url#1{{\tt\def~{\lower3.5pt\hbox{\char'176}}\def\_{\char'137}#1}}

\def\makeautorefname#1#2{\expandafter\def\csname#1autorefname\endcsname{#2}}
\makeautorefname{equation}{Equation}%
\makeautorefname{footnote}{footnote}%
\makeautorefname{item}{item}%
\makeautorefname{figure}{Figure}%
\makeautorefname{table}{Table}%
\makeautorefname{part}{Part}%
\makeautorefname{appendix}{Appendix}%
\makeautorefname{chapter}{Chapter}%
\makeautorefname{section}{Section}%
\makeautorefname{subsection}{Section}%
\makeautorefname{subsubsection}{Section}%
\makeautorefname{paragraph}{Paragraph}%
\makeautorefname{subparagraph}{Paragraph}%
\makeautorefname{theorem}{Theorem}%
\makeautorefname{theo}{Theorem}%
\makeautorefname{thm}{Theorem}%
\makeautorefname{addendum}{Addendum}%
\makeautorefname{addend}{Addendum}%
\makeautorefname{add}{Addendum}%
\makeautorefname{maintheorem}{Main theorem}%
\makeautorefname{mainthm}{Main theorem}%
\makeautorefname{corollary}{Corollary}%
\makeautorefname{claim}{Claim}%
\makeautorefname{corol}{Corollary}%
\makeautorefname{coro}{Corollary}%
\makeautorefname{cor}{Corollary}%
\makeautorefname{lemma}{Lemma}%
\makeautorefname{lemm}{Lemma}%
\makeautorefname{lem}{Lemma}%
\makeautorefname{sublemma}{Sublemma}%
\makeautorefname{sublem}{Sublemma}%
\makeautorefname{subl}{Sublemma}%
\makeautorefname{proposition}{Proposition}%
\makeautorefname{proposit}{Proposition}%
\makeautorefname{propos}{Proposition}%
\makeautorefname{propo}{Proposition}%
\makeautorefname{prop}{Proposition}%
\makeautorefname{property}{Property}
\makeautorefname{proper}{Property}
\makeautorefname{scholium}{Scholium}%
\makeautorefname{step}{Step}%
\makeautorefname{conjecture}{Conjecture}%
\makeautorefname{conject}{Conjecture}%
\makeautorefname{conj}{Conjecture}%
\makeautorefname{question}{Question}
\makeautorefname{questn}{Question}
\makeautorefname{quest}{Question}
\makeautorefname{ques}{Question}
\makeautorefname{qn}{Question}
\makeautorefname{definition}{Definition}%
\makeautorefname{defin}{Definition}%
\makeautorefname{defi}{Definition}%
\makeautorefname{def}{Definition}%
\makeautorefname{defn}{Definition}%
\makeautorefname{dfn}{Definition}%
\makeautorefname{notation}{Notation}
\makeautorefname{nota}{Notation}
\makeautorefname{notn}{Notation}
\makeautorefname{remark}{Remark}%
\makeautorefname{rema}{Remark}%
\makeautorefname{rem}{Remark}%
\makeautorefname{rmk}{Remark}%
\makeautorefname{rk}{Remark}%
\makeautorefname{remarks}{Remarks}%
\makeautorefname{rems}{Remarks}%
\makeautorefname{rmks}{Remarks}%
\makeautorefname{rks}{Remarks}%
\makeautorefname{example}{Example}%
\makeautorefname{examp}{Example}%
\makeautorefname{exmp}{Example}%
\makeautorefname{exmps}{Examples}%
\makeautorefname{exam}{Example}%
\makeautorefname{exa}{Example}%
\makeautorefname{ex}{Example}%
\makeautorefname{algorithm}{Algorith}%
\makeautorefname{algo}{Algorith}%
\makeautorefname{alg}{Algorith}%
\makeautorefname{axiom}{Axiom}%
\makeautorefname{axi}{Axiom}%
\makeautorefname{ax}{Axiom}%
\makeautorefname{case}{Case}%
\makeautorefname{claim}{Claim}%
\makeautorefname{clm}{Claim}%
\makeautorefname{assumption}{Assumption}%
\makeautorefname{assumpt}{Assumption}%
\makeautorefname{conclusion}{Conclusion}%
\makeautorefname{concl}{Conclusion}%
\makeautorefname{conc}{Conclusion}%
\makeautorefname{condition}{Condition}%
\makeautorefname{condit}{Condition}%
\makeautorefname{cond}{Condition}%
\makeautorefname{construction}{Construction}%
\makeautorefname{construct}{Construction}%
\makeautorefname{const}{Construction}%
\makeautorefname{cons}{Construction}%
\makeautorefname{criterion}{Criterion}%
\makeautorefname{criter}{Criterion}%
\makeautorefname{crit}{Criterion}%
\makeautorefname{exercise}{Exercise}%
\makeautorefname{exer}{Exercise}%
\makeautorefname{exe}{Exercise}%
\makeautorefname{problem}{Problem}%
\makeautorefname{problm}{Problem}%
\makeautorefname{probm}{Problem}%
\makeautorefname{prob}{Problem}%
\makeautorefname{solution}{Solution}%
\makeautorefname{soln}{Solution}%
\makeautorefname{sol}{Solution}%
\makeautorefname{summary}{Summary}%
\makeautorefname{summ}{Summary}%
\makeautorefname{sum}{Summary}%
\makeautorefname{operation}{Operation}%
\makeautorefname{oper}{Operation}%
\makeautorefname{observation}{Observation}%
\makeautorefname{observn}{Observation}%
\makeautorefname{obser}{Observation}%
\makeautorefname{obs}{Observation}%
\makeautorefname{ob}{Observation}%
\makeautorefname{convention}{Convention}%
\makeautorefname{convent}{Convention}%
\makeautorefname{conv}{Convention}%
\makeautorefname{cvn}{Convention}%
\makeautorefname{warning}{Warning}%
\makeautorefname{warn}{Warning}%
\makeautorefname{note}{Note}%
\makeautorefname{fact}{Fact}%

  \makeatletter
                   \let\c@lemma\c@theorem
                  \makeatother

\newtheorem{thm}{Theorem}[subsection]
\newtheorem{cor}[thm]{Corollary}
\newtheorem{prop}[thm]{Proposition}
\newtheorem{lem}[thm]{Lemma}
\theoremstyle{definition}
\newtheorem{defn}[thm]{Definition}

\newtheorem{rem}[thm]{Remark}

\newtheorem{notation}[thm]{Notation}

\makeatletter
\let\c@lem=\c@thm
\let\c@cor=\c@thm
\let\c@prop=\c@thm
\let\c@lem=\c@thm
\let\c@defn=\c@thm
\let\c@exmps=\c@thm
\let\c@rem=\c@thm
\let\c@warn=\c@thm
\let\c@claim=\c@thm
\let\c@quest=\c@thm
\makeatother

\numberwithin{equation}{subsection}

\newcommand{\U}{\mathbb{U}}
\newcommand{\M}{\mathsf{M}}
\newcommand{\Z}{\mathbb{Z}}

\DeclareSymbolFontAlphabet{\scr}{rsfs}

\def\quickop#1{\expandafter\newcommand\csname #1\endcsname{\operatorname{#1}}}
\quickop{Hom} \quickop{End} \quickop{Aut} \quickop{Tel} \quickop{Mic} 
\quickop{Ext} \quickop{Tor} \quickop{Cotor} \quickop{Id} \quickop{Coker} \quickop{Ker}
\quickop{Lim} \quickop{Colim} \quickop{Holim} \quickop{Hocolim}
\quickop{id} \quickop{tel} \quickop{mic} \quickop{coker} 
\quickop{colim} \quickop{holim} \quickop{hocolim} \quickop{im}

\DeclareMathOperator{\Tot}{Tot}

\newcommand{\K}{\mathbb{K}}

\newcommand{\can}{\mathrm{Can}}
\newcommand{\res}{\mathrm{res}}
\newcommand{\inj}{\mathrm{inj}}

\newcommand{\ve}{\varepsilon}

\newcommand\egal[2]{\overset {#1}{\underset {#2}\rightrightarrows }}

\newcommand{\WE}{\mathcal{W}}
\newcommand{\Fib}{\mathcal{F}}
\newcommand{\Cof}{\mathcal{C}}

\DeclareMathOperator{\LLP}{LLP}
\DeclareMathOperator{\RLP}{RLP} 

\newcommand{\sM}{\mathsf M}
\newcommand{\sN}{\mathsf N}
\newcommand{\sP}{\mathsf P}
\newcommand{\sfC}{\mathsf C}

\newcommand{\tilZ}{\widetilde{\mathbb{Z}}}

\newcommand{\spe}{\mathsf{Sp}^\Sigma}

\newcommand{\lra}{\longrightarrow}

\newcommand{\coalg}{\mathsf{coAlg}}

\newcommand{\sset}{\mathsf{sSet_*}}
\newcommand\adjunction[4]{\xymatrix{#1\ar @<1.18ex>[rr]^-{#3}&\perp&#2\ar @<1.18ex>[ll]^-{#4}}}

\definecolor{darkspringgreen}{rgb}{0.09, 0.45, 0.27}
\definecolor{darkterracotta}{rgb}{0.8, 0.31, 0.36}
	\definecolor{darkcoral}{rgb}{0.8, 0.36, 0.27}
	\definecolor{indiagreen}{rgb}{0.07, 0.53, 0.03}
	\definecolor{mountainmeadow}{rgb}{0.19, 0.73, 0.56}
	\definecolor{mountbattenpink}{rgb}{0.6, 0.48, 0.55}
	\definecolor{palatinatepurple}{rgb}{0.41, 0.16, 0.38}
	\definecolor{cinnamon}{rgb}{0.82, 0.41, 0.12}
	\definecolor{chocolate}{rgb}{0.82, 0.41, 0.12}
	
\newcommand{\sAb}{\mathsf{sAb}}

\DeclareMathOperator{\Map}{Map} 


%% file: coalg_new_no_reduced.bbl
\begin{thebibliography}{10}

\bibitem{AroneChing}
G.~Arone and M.~Ching.
\newblock A classification of {T}aylor towers of functors of spaces and
  spectra.
\newblock {\em Adv. Math.}, 272:471--552, 2015.

\bibitem{Banff1}
M.~Bayeh, K.~Hess, V.~Karpova, M.~K{\k{e}}dziorek, E.~Riehl, and B.~Shipley.
\newblock Left-induced model structures and diagram categories.
\newblock In {\em Women in topology: collaborations in homotopy theory}, volume
  641 of {\em Contemp. Math.}, pages 49--81. Amer. Math. Soc., Providence, RI,
  2015.

\bibitem{BlomquistHarper1}
J.~R. Blomquist and J.~E. Harper.
\newblock An integral chains analog of {Q}uillen's rational homotopy theory
  equivalence.
\newblock arXiv:1611.04157.

\bibitem{BlomquistHarper2}
J.~R. Blomquist and J.~E. Harper.
\newblock Iterated suspension spaces and higher {F}reudenthal suspension.
\newblock arXiv:1612.08622.

\bibitem{BlomquistHarper3}
J.~R. Blomquist and J.~E. Harper.
\newblock Suspension spectra and higher stabilization.
\newblock arXiv:1612.08623.

\bibitem{BorceuxHandbook2}
F.~Borceux.
\newblock {\em Handbook of categorical algebra. 2}, volume~51 of {\em
  Encyclopedia of Mathematics and its Applications}.
\newblock Cambridge University Press, Cambridge, 1994.
\newblock Categories and structures.

\bibitem{BourkeGarner}
J.~Bourke and R.~Garner.
\newblock Algebraic weak factorisation systems {I}: {A}ccessible {AWFS}.
\newblock {\em J. Pure Appl. Algebra}, 220(1):108--147, 2016.

\bibitem{BousfieldKan}
A.~K. Bousfield and D.~M. Kan.
\newblock {\em Homotopy limits, completions and localizations}.
\newblock Lecture Notes in Mathematics, Vol. 304. Springer-Verlag, Berlin,
  1972.

\bibitem{ChingRiehl}
M.~Ching and E.~Riehl.
\newblock Coalgebraic models for combinatorial model categories.
\newblock {\em Homology Homotopy Appl.}, 16(2):171--184, 2014.

\bibitem{GKR}
R.~Garner, M.~K\k{e}dziorek, and E.~Riehl.
\newblock Lifting accessible model structures.
\newblock arXiv:1802.09889.

\bibitem{GoerssJardine}
P.~G. Goerss and J.~F. Jardine.
\newblock {\em Simplicial homotopy theory}, volume 174 of {\em Progress in
  Mathematics}.
\newblock Birkh\"auser Verlag, Basel, 1999.

\bibitem{hess:descent}
K.~Hess.
\newblock A general framework for homotopic descent and codescent.
\newblock arXiv:1001.1556.

\bibitem{HKRS}
K.~Hess, M.~K\k{e}dziorek, E.~Riehl, and B.~Shipley.
\newblock A necessary and sufficient condition for induced model structures.
\newblock {\em J. Topol.}, 10(2):324--369, 2017.

\bibitem{Coalgebras_Comonad_Post}
K.~Hess and B.~Shipley.
\newblock The homotopy theory of coalgebras over a comonad.
\newblock {\em Proc. Lond. Math. Soc. (3)}, 108(2):484--516, 2014.

\bibitem{HoveySSSymSpectra}
M.~Hovey, B.~Shipley, and J.~Smith.
\newblock Symmetric spectra.
\newblock {\em J. Amer. Math. Soc.}, 13(1):149--208, 2000.

\bibitem{joyal-tierney}
A.~Joyal and M.~Tierney.
\newblock Quasi-categories vs {S}egal spaces.
\newblock In {\em Categories in algebra, geometry and mathematical physics},
  volume 431 of {\em Contemp. Math.}, pages 277--326. Amer. Math. Soc.,
  Providence, RI, 2007.

\bibitem{MacLane}
S.~Mac~Lane.
\newblock {\em Categories for the working mathematician}, volume~5 of {\em
  Graduate Texts in Mathematics}.
\newblock Springer-Verlag, New York, second edition, 1998.

\bibitem{RosickyAccessibleModelCats}
J.~Rosick\'y.
\newblock Accessible model categories.
\newblock {\em Appl. Categ. Structures}, 25(2):187--196, 2017.

\end{thebibliography}
